\RequirePackage{fix-cm}
\documentclass[smallextended,referee,envcountsect]{svjour3}       
\smartqed  

\usepackage{graphicx}
\usepackage{amsmath, ntheorem}
\usepackage{amssymb}
\usepackage{stackrel}
\usepackage{marvosym, enumerate}
\usepackage{amstext, amscd, latexsym}
\usepackage{amsfonts, ragged2e}
\usepackage{mathrsfs, pgfplots, enumerate, setspace}
\usepackage{subfigure}
\usepackage{cases}

\smartqed
\usepackage{cite}
\usepackage{amstext, graphicx, amscd, latexsym, amssymb}
\usepackage{amsfonts, ragged2e}
\usepackage{pgfplots, setspace}
\usepackage{geometry}
\usepackage{latexsym,bm}
\usepackage{float}
\usepackage{epstopdf,caption}
 
\usepackage{amsfonts,amsmath,amsbsy,amssymb,amscd,mathrsfs}
\usepackage{graphicx}
\usepackage{epstopdf}
\usepackage{algorithmic} 
\usepackage{graphicx,epstopdf}

 \newcommand{\nm}[1]{\left\lVert {#1} \right\rVert}
 
 \newcommand{\dual}[1]{\left\langle {#1} \right\rangle}

\usepackage[figuresright]{rotating}
\usepackage[sort&compress,numbers]{natbib}
\usepackage[ruled,linesnumbered]{algorithm2e} 
\usepackage[colorlinks,linkcolor=blue,anchorcolor=blue,citecolor=blue,urlcolor=blue]{hyperref}

\journalname{}

\begin{document}

\title{On the Convergence of Newton-type Proximal Gradient Method for Multiobjective Optimization Problems}


\author{Jian Chen \and Xiaoxue Jiang \and Liping Tang \and  Xinmin Yang  }

\institute{J. Chen \at National  Center  for  Applied  Mathematics in Chongqing, Chongqing Normal University, Chongqing 401331, China, and School of Mathematical Sciences, University of Electronic Science and Technology of China, Chengdu, Sichuan 611731, China\\
                    \href{mailto:chenjian_math@163.com}{chenjian\_math@163.com}\\
                    X.X. Jiang \at College of Mathematics, Sichuan University,
                    Chengdu, 610065, China\\
                    \href{xxjiangmath@126.com}{xxjiangmath@126.com}\\
                   L.P. Tang \at National Center for Applied Mathematics in Chongqing, and School of Mathematical Sciences,  Chongqing Normal University, Chongqing 401331, China\\
                   \href{mailto:tanglipings@163.com}{tanglipings@163.com}\\
        \Letter X.M. Yang \at National Center for Applied Mathematics in Chongqing, and School of Mathematical Sciences,  Chongqing Normal University, Chongqing 401331, China\\
        \href{mailto:xmyang@cqnu.edu.cn}{xmyang@cqnu.edu.cn} \\}

\date{Received: date / Accepted: date}

\maketitle

\begin{abstract}
	In a recent study, Ansary (Optim Methods Softw 38(3):570-590,2023) proposed a Newton-type proximal gradient method for nonlinear multiobjective optimization problems (NPGMO). However, the favorable convergence properties typically associated with Newton-type methods were not established for NPGMO in Ansary's work. In response to this gap, we develop a straightforward framework for analyzing the convergence behavior of the NPGMO. Specifically, under the assumption of strong convexity, we demonstrate that the NPGMO enjoys quadratic termination, superlinear convergence, and quadratic convergence for problems that are quadratic, twice continuously differentiable and twice Lipschitz continuously differentiable, respectively. 

\end{abstract}

\section{Introduction}
A multiobjective composite optimization problems (MCOPs) can be formulated as follows:
\begin{align*}
	\min\limits_{x\in\mathbb{R}^{n}} F(x), \tag{MCOP}\label{MCOP}
\end{align*}
where $F:\mathbb{R}^{n}\rightarrow\mathbb{R}^{m}$ is a vector-valued function. Each component $F_{i}$, $i=1,2,...,m$, is defined by
$$F_{i}:=f_{i}+g_{i},$$ 
where $f_{i}$ is continuously differentiable and $g_{i}$ is proper convex and lower semicontinuous but not necessarily differentiable. 
\par Over the past two decades, {\it descent methods} have garnered increasing attention within the multiobjective optimization community (see, e.g., \cite{AP2021,BI2005,CL2016,CTY2023,FS2000,FD2009,FV2016,GI2004,LP2018,MP2018,MP2019,P2014,QG2011,TFY2019} and references therein). These methods generate descent directions by solving subproblems, eliminating the need for predefined parameters. As far as we know, the study of multiobjective gradient descent methods can be traced back to the pioneering efforts of Mukai \cite{M1980} as well as the seminal work of Fliege and Svaiter \cite{FS2000}.
\par Recently, Tanabe et al. \cite{TFY2019} proposed a proximal gradient method for MCOPs (PGMO), and their subsequent convergence rates analysis \cite{TFY2023} revealed that the PGMO exhibits slow convergence when dealing with ill-conditioned problems. Notably, most of multiobjective first-order methods are prone to sensitivity concerning problem conditioning. In response to this challenge, Fliege et al. \cite{FD2009} leveraged a quadratic model to better capture the problem's geometry, and proposed the Newton's method for MOPs (NMO). Parallel to its single-objective counterpart, the NMO demonstrates locally superlinear and quadratic convergence under standard assumptions. Building upon this foundation, Ansary recently adopted the idea of Fliege et al. \cite{FD2009}, and extended it to derive the Newton-type proximal gradient method for MOPs (NPGMO) \cite{A2023}. However, the appealing convergence properties associated with Newton's method were not established within Ansary's work. It is worth mentioning that Newton-type proximal gradient method for SOPs \cite{LSS2014} exhibits convergence akin to Newton's method. This naturally raises the question: Does NPGMO enjoy the same desirable convergence properties as NMO?
\par The primary objective of this paper is to provide an affirmative answer to the question. The contributions of this study can be summarized as follows:
\par (i) We demonstrate that the NPGMO finds an exact Pareto solution in one iteration for strongly convex quadratic problems.
\par (ii) Given the assumptions that $f_{i}$ is both strongly convex and twice continuously differentiable for $i=1,2,...,m$, we prove that NPGMO is strong convergence, and the local convergence rate is superlinear. Additionally, by further assuming that each $f_{i}$ is twice Lipschitz continuously differentiable, the local convergence rate of NPGMO is elevated to a quadratic rate.
\par (iii) We derive a simplified analysis framework to elucidate the appealing convergence properties of NPGMO, which can be linked back to the properties of NMO. The proofs are grounded in a modified fundamental inequality, a common tool employed in the convergence analysis of first-order methods.

\par The paper is organized as follows. In section \ref{sec2}, we present some necessary notations and definitions that will be used later. Section \ref{sec3} revisits the descriptions and highlights certain properties of NPGMO. The convergence analysis of NPGMO is detailed in Section \ref{sec4}. Lastly, we draw our conclusions in the final section of the paper.

\section{Preliminaries}\label{sec2}
Throughout this paper, the $n$-dimensional Euclidean space $\mathbb{R}^{n}$ is equipped with the inner product $\langle\cdot,\cdot\rangle$ and the induced norm $\|\cdot\|$. Denote $\mathbb{S}^{n}_{++}(\mathbb{S}^{n}_{+})$ the set of symmetric (semi-)positive definite matrices in $\mathbb{R}^{n\times n}$. For a positive definite matrix $H$, the notation $\|x\|_{H}=\sqrt{\langle x,Hx \rangle}$ is used to represent the norm induced by $H$ on vector $x$.  Additionally, we define
$$F^{\prime}_{i}(x;d):=\lim\limits_{t\downarrow0}\frac{F_{i}(x+td)-F_{i}(x)}{t}$$
the directional derivative of $F_{i}$ at $x$ in the direction $d$. 
%

For simplicity, we utilize the notation $[m]:={1,2,...,m}$, and define $$\Delta_{m}:=\left\{\lambda:\sum\limits_{i\in[m]}\lambda_{i}=1,\lambda_{i}\geq0,\ i\in[m]\right\}$$ to represent the $m$-dimensional unit simplex. To prevent any ambiguity, we establish the order $\preceq(\prec)$ in $\mathbb{R}^{m}$ as follows: $$u\preceq(\prec)v~\Leftrightarrow~v-u\in\mathbb{R}^{m}_{+}(\mathbb{R}^{m}_{++}),$$
and in $\mathbb{S}^{n}$ as:
$$U\preceq(\prec)V~\Leftrightarrow~V-U\in\mathbb{S}^{n}_{+}(\mathbb{S}^{n}_{++}).$$
\par In the following, we introduce the concepts of optimality for (\ref{MCOP}) in the Pareto sense. 
\vspace{2mm}
\begin{definition}\label{def1}
	A vector $x^{\ast}\in\mathbb{R}^{n}$ is called Pareto solution to (\ref{MCOP}), if there exists no $x\in\mathbb{R}^{n}$ such that $F(x)\preceq F(x^{\ast})$ and $F(x)\neq F(x^{\ast})$.
\end{definition}
\vspace{2mm}
\begin{definition}\label{def2}
	A vector $x^{\ast}\in\mathbb{R}^{n}$ is called weakly Pareto solution to (\ref{MCOP}), if there exists no $x\in\mathbb{R}^{n}$ such that $F(x)\prec F(x^{\ast})$.
\end{definition}
\vspace{2mm}
\begin{definition}\label{def3}
	A vector $x^{\ast}\in\mathbb{R}^{n}$ is called  Pareto critical point of (\ref{MCOP}), if
	$$\max\limits_{i\in[m]}F_{i}^{\prime}(x^{*};d)\geq0,~\forall d\in\mathbb{R}^{n}.$$
\end{definition}

\par From Definitions \ref{def1} and \ref{def2}, it is evident that Pareto solutions are always weakly Pareto solutions. The following lemma shows the relationships among the three concepts of Pareto optimality.
\vspace{2mm}
\begin{lemma}[Theorem 3.1 of \cite{FD2009}] The following statements hold.
	\begin{itemize}
		\item[$\mathrm{(i)}$]  If $x\in\mathbb{R}^{n}$ is a weakly Pareto solution to (\ref{MCOP}), then $x$ is Pareto critical point.
		\item[$\mathrm{(ii)}$] Let every component $F_{i}$ of $F$ be convex. If $x\in\mathbb{R}^{n}$ is a Pareto critical point of (\ref{MCOP}), then $x$ is weakly Pareto solution.
		\item[$\mathrm{(iii)}$] Let every component $F_{i}$ of $F$ be strictly convex. If $x\in\mathbb{R}^{n}$ is a Pareto critical point of (\ref{MCOP}), then $x$ is Pareto solution.
	\end{itemize}
\end{lemma}
\vspace{2mm}
\begin{definition}
	A twice continuously differentiable function $h$ is $\mu$-strongly convex if $$\mu I\preceq\nabla^{2}h(x)$$ holds for all $x\in\mathbb{R}^{n}$.
\end{definition}

\section{Newton-type proximal gradient method for MCOPs}\label{sec3}In this section, we revisit the multiobjective proximal Newton-type method, which was originally introduced by by Ansary \cite{A2023}. 
\subsection{Newton-type proximal gradient method}
A multiobjective Newton-type proximal search direction corresponds to the unique optimal solution of the following subproblem:
\begin{equation}\label{subp}
	\mathop{\min}\limits_{d\in\mathbb{R}^{n}}\max\limits_{i\in[m]}\left\{\langle\nabla f_{i}(x),d\rangle+g_{i}(x+d)-g_{i}(x)+\frac{1}{2}\dual{d,\nabla^{2}f_{i}(x)d}\right\},
\end{equation}
namely,
$$d(x):=\mathop{\arg\min}\limits_{d\in\mathbb{R}^{n}}\max\limits_{i\in[m]}\left\{\langle\nabla f_{i}(x),d\rangle+g_{i}(x+d)-g_{i}(x)+\frac{1}{2}\dual{d,\nabla^{2}f_{i}(x)d}\right\}.$$
The optimal value of the subproblem is denoted by
$$\theta(x):=\min\limits_{d\in\mathbb{R}^{n}}\max\limits_{i\in[m]}\left\{\langle\nabla f_{i}(x),d\rangle+g_{i}(x+d)-g_{i}(x)+\frac{1}{2}\dual{d,\nabla^{2}f_{i}(x)d}\right\}.$$
For simplicity, we denote by $$h_{\lambda}(x):=\sum\limits_{i\in[m]}\lambda_{i}h_{i}(x),$$
$$\nabla h_{\lambda}(x):=\sum\limits_{i\in[m]}\lambda_{i}\nabla h_{i}(x),$$
$$\nabla^{2}h_{\lambda}(x):=\sum\limits_{i\in[m]}\lambda_{i}\nabla^{2}h_{i}(x).$$
By Sion's minimax theorem \citep{S1958}, there exists $\lambda^{k}\in\Delta_{m}$ such that
$$d^{k}=\mathop{\arg\min}\limits_{d\in\mathbb{R}^{n}}\left\{\langle\nabla f_{\lambda^{k}}(x^{k}),d\rangle+g_{\lambda^{k}}(x^{k}+d)-g_{\lambda^{k}}(x^{k})+\frac{1}{2}\dual{d,\nabla^{2}f_{\lambda^{k}}(x^{k})d}\right\},$$
we use the first-order optimality condition to get
\begin{equation}\label{subgrad}
	-\nabla f_{\lambda^{k}}(x^{k})-\nabla^{2}f_{\lambda^{k}}(x^{k})d^{k}\in\partial g_{\lambda^{k}}(x^{k}+d^{k}).
\end{equation}

The following lemma presents several properties of $d(x)$. 
\begin{lemma}[Lemma 3.2 and Theorem 3.1 of \cite{A2023}]
	Suppose $f_{i}$ is strictly convex function for all $i\in[m]$. Let $d(x)$ be the optimal solution of (\ref{subp}). Then the following statements hold.
	\begin{itemize}
		\item[$\mathrm{(i)}$] $x\in\mathbb{R}^{n}$ is a Pareto critical point of (\ref{MCOP}) if and only if $d(x)=0$. 
		\item[$\mathrm{(ii)}$] If $\{x^{k}\}$ converges $x^{*}$, and $\{d(x^{k})\}$ converges $d^{*}$, then $d(x^{*})=d^{*}$.
	\end{itemize}
\end{lemma}

\begin{lemma}[Theorem 3.2 of \cite{A2023}]\label{lem3.2}
	Suppose $f_{i}$ is strongly convex function with module $\mu>0$ for $i\in[m]$. Then 
	$$\theta(x)\leq-\frac{\mu}{2}\nm{d(x)}^{2}.$$
\end{lemma}
\par The multiobjective Newton-type proximal gradient with line search is described as follows:
\begin{algorithm}  
	\caption{{\ttfamily{Newton-type\_proximal\_gradient\_method\_for\_MCOPs \cite{A2023}}}}\label{npgmo}
	\LinesNumbered  
	\KwData{$x^{0}\in\mathbb{R}^{n},~\epsilon>0,~\sigma,\gamma\in(0,1)$}
	\For{$k=0,...$}{Compute $d^{k}$ and $\theta^{k}$ by solving subproblem (\ref{subp}) with $x=x^{k}$\\
		\eIf{$\|d^{k}\|<\epsilon$}{ {\bf{return}} approximated Pareto critical point $x^{k}$ }{Compute the stepsize $t_{k}\in(0,1]$ as the maximum of $$T_{k}:=\{\gamma^{j}:j\in\mathbb{N},~F_{i}(x^{k}+t_{k}d^{k})-F_{i}(x^{k})\leq \gamma^{j}\sigma\theta^{k},~i\in[m]\}$$\\
			Update	$x^{k+1}:= x^{k}+t_{k}d^{k}$}}  
\end{algorithm}
\section{Convergence analysis of NPGMO}\label{sec4}
As we know, the Newton-type proximal gradient method for SOPs exhibits desirable convergence behavior of Newton-type methods for minimizing smooth functions, see \cite{LSS2014}. Naturally, this prompts the question: Does NPGMO exhibit similar desirable convergence properties to those of NMO when applied to minimizing smooth functions?

\subsection{Strong convergence}
As evident from Algorithm \ref{npgmo}, it terminates either with a Pareto critical point in a finite number of iterations or generates an infinite sequence of points. In the forthcoming analysis, we will assume that Algorithm \ref{npgmo} generates an infinite sequence of noncritical points. In \cite{A2023}, Ansary analyzed the global convergence of NPGMO.
\begin{theorem}[Theorem 4.1 of \cite{A2023}]\label{T1}
	Suppose $f_{i}$ is strongly convex with module $\mu>0$, its gradient is Lipschitz continuous with constant $L_{1}$ for all $i\in[m]$, and the level set $\mathcal{L}_{F}(x^{0}):=\{x:F(x)\preceq F(x^{0})\}$ is bounded. Let $\{x^{k}\}$ be the sequence produced by Algorithm \ref{npgmo}. Then every accumulation point of $\{x^{k}\}$ is a Pareto critical point of (\ref{MCOP}).
\end{theorem}

We can derive the strong convergence of NPGMO.
\begin{theorem}\label{T2}
	Suppose $f_{i}$ is strongly convex with module $\mu>0$ for all $i\in[m]$. Let $\{x^{k}\}$ be the sequence produced by Algorithm \ref{npgmo}. Then $\{x^{k}\}$ converges to some Pareto point $x^{*}$.
\end{theorem}
\begin{proof}
	Since $f_i$ is strongly convex and $g_{i}$ is convex for $i\in[m]$, this, together with the continuity of $f_{i}$ and lower semi-continuity of $g_{i}$ for $i\in[m]$, yields the level set $\mathcal{L}_{F}(x^{0})\subset\{x:F_{i}(x)\preceq F_{i}(x^{0})\}$ is compact, and any Pareto critical point is a Pareto point. Moreover, we can infer $\{x^{k}\}\subset\mathcal{L}_{F}(x^{0})$ due to the fact that $\{F(x^{k})\}$ is decreasing. With the compactness of $\mathcal{L}_{F}(x^{0})$, it follows that $\{x^{k}\}$ has an accumulation point $x^{*}$ and $\lim\limits_{k\rightarrow\infty}F(x^{k})=F(x^{*})$ . By applying a similar argument as presented in the proof of \cite[Theorem 4.2]{TFY2019}, it can be concluded that $x^{*}$ is a Pareto critical point (Pareto point). Next, we proceed to prove the uniqueness of $x^{*}$. Suppose the contrary that there exists another distinct accumulation point $x^{*}_{1}$. By the strong convexity of $F_{i}$ for $i\in[m]$, the following inequality holds: $$F(sx^{*}+(1-s)x^{*}_{1})\prec sF(x^{*})+(1-s)F(x^{*}_{1})=F(x^{*}),$$ 
	where the equality follows from the convergence of $\{F(x^{k})\}$.
	However, this contradicts the fact that $x^{*}$ is a Pareto point. The uniqueness of accumulation point of $\{x^{k}\}$ implies that $\{x^{k}\}$ converges to $x^{*}$. 
\end{proof}

\subsection{Quadratic termination}
In this section, we will delve into the analysis of the quadratic termination property of Algorithm \ref{npgmo}. Prior to presenting the outcome, we lay the foundation by establishing the subsequent modified fundamental inequality.
\begin{proposition}[modified fundamental inequality]
	Suppose $f_{i}$ is strictly convex for all $i\in[m]$. Let $\{x^{k}\}$ be the sequence produced by Algorithm \ref{npgmo}. If $t_k=1$, then there exists $\lambda^{k}\in\Delta_{m}$ such that
	\begin{equation}\label{ineq}
		\begin{aligned}
			&~~~~F_{\lambda^{k}}(x^{k+1}) - F_{\lambda^{k}}(x)\\
			&\leq\frac{1}{2}\nm{x^{k}-x}^{2}_{\nabla^{2}f_{\lambda^{k}}(x^{k})}-\frac{1}{2}\nm{x^{k+1}-x^{k}}^{2}_{\nabla^{2}f_{\lambda^{k}}(x^{k})}-\frac{1}{2}\nm{x^{k+1}-x}^{2}_{\nabla^{2}f_{\lambda^{k}}(x^{k})}\\
			&~~~~+\dual{x^{k+1}-x^{k},\int_{0}^{1}\int_{0}^{1}\nabla^{2}f_{\lambda^{k}}(x^{k}+st(x^{k+1}-x^{k}))ds(t(x^{k+1}-x^{k}))dt}\\
			&~~~~-\dual{x-x^{k},\int_{0}^{1}\int_{0}^{1}\nabla^{2}f_{\lambda^{k}}(x^{k}+st(x-x^{k}))ds(t(x-x^{k}))dt}.
		\end{aligned}
	\end{equation}
\end{proposition}
\begin{proof}
	From the twice continuity of $f_{i}$, we can deduce that
	\begin{small}
	\begin{align*}
		&~~~~F_{i}(x^{k+1}) - F_{i}(x)\\
		&= (f_{i}(x^{k+1})-f_{i}(x^{k})) - (f_{i}(x)-f_{i}(x^{k}))+g_{i}(x^{k+1})-g_{i}(x)\\
		&=\dual{\nabla f_{i}(x^{k}),x^{k+1}-x^{k}}+\dual{x^{k+1}-x^{k},\int_{0}^{1}\int_{0}^{1}\nabla^{2}f_{i}(x^{k}+st(x^{k+1}-x^{k}))ds(t(x^{k+1}-x^{k}))dt}\\
		&~~~~+\dual{\nabla f_{i}(x^{k}),x^{k}-x}-\dual{x-x^{k},\int_{0}^{1}\int_{0}^{1}\nabla^{2}f_{i}(x^{k}+st(x-x^{k}))ds(t(x-x^{k}))dt}+g_{i}(x^{k+1})-g_{i}(x)\\
		&=\dual{\nabla f_{i}(x^{k}),x^{k+1}-x}+\dual{x^{k+1}-x^{k},\int_{0}^{1}\int_{0}^{1}\nabla^{2}f_{i}(x^{k}+st(x^{k+1}-x^{k}))ds(t(x^{k+1}-x^{k}))dt}\\
		&~~~~-\dual{x-x^{k},\int_{0}^{1}\int_{0}^{1}\nabla^{2}f_{i}(x^{k}+st(x-x^{k}))ds(t(x-x^{k}))dt}+g_{i}(x^{k+1})-g_{i}(x).
	\end{align*}
\end{small}
	On the other hand, from (\ref{subgrad}), we have $-\nabla f_{\lambda^{k}}(x^{k})-\nabla^{2}f_{\lambda^{k}}(x^{k})d^{k}\in\partial g_{\lambda^{k}}(x^{k}+d^{k}),~\lambda^{k}\in\Delta_{m}$. This, together with the fact that $t_k=1$, implies
	\begin{align*}
		g_{\lambda^{k}}(x^{k+1})-g_{\lambda^{k}}(x)&\leq\dual{-\nabla f_{\lambda^{k}}(x^{k})-\nabla^{2}f_{\lambda^{k}}(x^{k})d^{k},x^{k+1}-x}\\
		&=\dual{-\nabla f_{\lambda^{k}}(x^{k})-\nabla^{2}f_{\lambda^{k}}(x^{k})(x^{k+1}-x^{k}),x^{k+1}-x}.
	\end{align*}
	We use the last two relations to get
	\begin{small}
	\begin{align*}
		&~~~~F_{\lambda^{k}}(x^{k+1}) - F_{\lambda^{k}}(x)\\
		&\leq\dual{\nabla^{2}f_{\lambda^{k}}(x^{k})(x^{k}-x^{k+1}),x^{k+1}-x}+\dual{x^{k+1}-x^{k},\int_{0}^{1}\int_{0}^{1}\nabla^{2}f_{\lambda^{k}}(x^{k}+st(x^{k+1}-x^{k}))ds(t(x^{k+1}-x^{k}))dt}\\
		&~~~~-\dual{x-x^{k},\int_{0}^{1}\int_{0}^{1}\nabla^{2}f_{\lambda^{k}}(x^{k}+st(x-x^{k}))ds(t(x-x^{k}))dt}\\
		&=\frac{1}{2}\nm{x^{k}-x}^{2}_{\nabla^{2}f_{\lambda^{k}}(x^{k})}-\frac{1}{2}\nm{x^{k+1}-x^{k}}^{2}_{\nabla^{2}f_{\lambda^{k}}(x^{k})}-\frac{1}{2}\nm{x^{k+1}-x}^{2}_{\nabla^{2}f_{\lambda^{k}}(x^{k})}\\
		&~~~~+\dual{x^{k+1}-x^{k},\int_{0}^{1}\int_{0}^{1}\nabla^{2}f_{\lambda^{k}}(x^{k}+st(x^{k+1}-x^{k}))ds(t(x^{k+1}-x^{k}))dt}\\
		&~~~~-\dual{x-x^{k},\int_{0}^{1}\int_{0}^{1}\nabla^{2}f_{\lambda^{k}}(x^{k}+st(x-x^{k}))ds(t(x-x^{k}))dt},
	\end{align*}
\end{small}
	where the equality comes from the three points lemma \cite{CT1993}. This completes the proof.
\end{proof}

\begin{theorem}
	Suppose $f_{i}$ is strongly convex quadratic problem for $i\in[m]$. Let $\{x^{k}\}$ be the sequence produced by Algorithm \ref{npgmo}. Then $x^{k+1}=x^{*}$, i.e.,  Algorithm \ref{npgmo} finds an exact Pareto point of (\ref{MCOP}) in one iteration.
\end{theorem}
\begin{proof}
	First, we aim to establish that $t_{k}=1$ for all $k$. 
	Since $f_{i}$ is strongly convex quadratic problem for $i\in[m]$, there exists a positive definite matrix $A_{i}$ such that $\nabla^{2}f_{i}(x)=A_{i}$ for all $x\in\mathbb{R}^{n}$. This leads us to the conclusion that
	\begin{align*}
		F_{i}(x^{k}+d^{k})-F_{i}(x^{k})&=\dual{\nabla f_{i}(x^{k}),d^{k}}+g_{i}(x^{k}+d^{k})-g_{i}(x^{k})+\frac{1}{2}\dual{d^{k},A_{i}d^{k}}\\
		&\leq\max\limits_{i\in[m]}\left\{\dual{\nabla f_{i}(x^{k}),d^{k}}+g_{i}(x^{k}+d^{k})-g_{i}(x^{k})+\frac{1}{2}\dual{d^{k},A_{i}d^{k}}\right\}\\
		&=\theta^{k}\\
		&\leq\sigma\theta^{k}.
	\end{align*}
	Thus, $t_{k}=1$ for all $k$. For all $x\in\mathbb{R}^{n}$, we use relation (\ref{ineq}) to get
	\begin{align*}
		F_{\lambda^{k}}(x^{k+1}) - F_{\lambda^{k}}(x)&\leq\frac{1}{2}\nm{x^{k}-x}^{2}_{A_{\lambda^{k}}}-\frac{1}{2}\nm{x^{k+1}-x^{k}}^{2}_{A_{\lambda^{k}}}-\frac{1}{2}\nm{x^{k+1}-x}^{2}_{A_{\lambda^{k}}}\\
		&~~~~+\dual{x^{k+1}-x^{k},\int_{0}^{1}A_{\lambda^{k}}t(x^{k+1}-x^{k})dt}-\dual{x-x^{k},\int_{0}^{1}A_{\lambda^{k}}t(x-x^{k})dt}\\
		&=-\frac{1}{2}\nm{x^{k+1}-x}^{2}_{A_{\lambda^{k}}},
	\end{align*}
	where $A_{\lambda^{k}}:=\sum\limits_{i\in[m]}\lambda^{k}_{i}A_{i}$. Applying Theorem \ref{T2}, we can assert the existence of $x^{*}$ such that $F(x^{*})\preceq F(x^{k})$ for all $k$. Substituting $x=x^{*}$ into above inequality, we have
	$$F_{\lambda^{k}}(x^{k+1}) - F_{\lambda^{k}}(x^{*})\leq-\frac{1}{2}\nm{x^{k+1}-x^{*}}^{2}_{A_{\lambda^{k}}}.$$
	Combining this with the fact that $F(x^{*})\preceq F(x^{k})$ leads to the following inequality
	$$\frac{1}{2}\nm{x^{k+1}-x^{*}}^{2}_{A_{\lambda^{k}}}\leq0.$$
	Considering that $A_{i}$ is a positive definite matrix for $i\in[m]$, it follows that $x^{k+1}=x^{*}$. This completes the proof.
\end{proof}
\par By setting $g_{i}(x)=0$, for $i\in[m]$, the quadratic termination property also applies to the NMO as presented in \cite{FD2009}.
\begin{corollary}
	Suppose $F_{i}$ is strongly convex quadratic problem for $i\in[m]$. Let $\{x^{k}\}$ be the sequence produced by the NMO. Then $x^{k+1}=x^{*}$, i.e.,  the NMO finds an exact Pareto point in one iteration.
\end{corollary}
\subsection{Local superlinear convergence}
Next, we give sufficient conditions for local superlinear convergence.
\begin{theorem}\label{T4}
	Suppose $f_{i}$ is strongly convex with module $\mu>0$, and its Hessian is continuous for $i\in[m]$. Let $\{x^{k}\}$ be the sequence produced by Algorithm \ref{npgmo}. Then, for any $0<\epsilon\leq(1-\sigma)\mu$, there exists $K_{\epsilon}>0$ such that $$\nm{x^{k+1}-x^{*}}\leq\sqrt{\frac{\epsilon(1+\tau_{k}^{2})}{\mu}}\nm{x^{k}-x^{*}}$$ holds for all $k\geq K_{\epsilon}$, where  $\tau_{k}:=\frac{\nm{x^{k+1}-x^{k}}}{\nm{x^{k}-x^{*}}}\in\left[\frac{\mu-\sqrt{2\mu\epsilon-\epsilon^{2}}}{\mu-\epsilon},\frac{\mu+\sqrt{2\mu\epsilon-\epsilon^{2}}}{\mu-\epsilon}\right]$. Furthermore, the sequence $\{x^{k}\}$ converges superlinearly to $x^{*}$.
\end{theorem}
\begin{proof}
	Referring to the arguments presented in the proof of Theorem \ref{T2}, we conclude that $\{x^{k}\}$ converges to a certain Pareto point $x^{*}$ and $\{d^{k}\}$ converges to $0$. This implies that for any $r>0$, there exists $K_{r}>0$ such that $x^{k},x^{k}+d^{k}\in B[x^{*},r]$ for all $k\geq K_{r}$. Given that $\nabla^{2}f_{i}$ is continuous, it follows that $\nabla^{2}f_{i}$ is uniformly continuous on the compact set $B[x^{*},r]$. For any $0<\epsilon\leq(1-\sigma)\mu$, there exists $K^{1}_{\epsilon}\geq K_{r}$ such that, for all $k\geq K^{1}_{\epsilon}$,
	\begin{align*}
		F_{i}(x^{k}+d^{k})-F_{i}(x^{k})&\leq\dual{\nabla f_{i}(x^{k}),d^{k}}+g_{i}(x^{k}+d^{k})-g_{i}(x^{k})+\frac{1}{2}\dual{d^{k},\nabla^{2}f_{i}(x^{k})d^{k}}+\frac{\epsilon}{2}\|d^{k}\|^{2}\\
		&\leq\max\limits_{i\in[m]}\left\{\dual{\nabla f_{i}(x^{k}),d^{k}}+g_{i}(x^{k}+d^{k})-g_{i}(x^{k})+\frac{1}{2}\dual{d^{k},\nabla^{2}f_{i}(x^{k})d^{k}}\right\}+\frac{\epsilon}{2}\|d^{k}\|^{2}\\
		&=\theta^{k}+\frac{\epsilon}{2}\|d^{k}\|^{2}\\
		&=\sigma\theta^{k}+(1-\sigma)\theta^{k}+\frac{\epsilon}{2}\|d^{k}\|^{2}\\
		&\leq\sigma\theta^{k},
	\end{align*}
	where the last inequality is due to the facts that $\theta^{k}\leq-\frac{\mu}{2}\nm{d^{k}}^{2}$ (Lemma \ref{lem3.2}) and $\epsilon\leq(1-\sigma)\mu$. Consequently, we can deduce that $t_{k}=1$ for all $k\geq K^{1}_{\epsilon}$.
	Substituting $x=x^{*}$ into (\ref{ineq}), we obtain
	\begin{equation}\label{e4}
		\begin{aligned}
			0&\leq F_{\lambda^{k}}(x^{k+1}) - F_{\lambda^{k}}(x^{*})\\
			&\leq\frac{1}{2}\nm{x^{k}-x^{*}}^{2}_{\nabla^{2}f_{\lambda^{k}}(x^{k})}-\frac{1}{2}\nm{x^{k+1}-x^{k}}^{2}_{\nabla^{2}f_{\lambda^{k}}(x^{k})}-\frac{1}{2}\nm{x^{k+1}-x^{*}}^{2}_{\nabla^{2}f_{\lambda^{k}}(x^{k})}\\
			&~~~~+\dual{x^{k+1}-x^{k},\int_{0}^{1}\int_{0}^{1}\nabla^{2}f_{\lambda^{k}}(x^{k}+st(x^{k+1}-x^{k}))ds(t(x^{k+1}-x^{k}))dt}\\
			&~~~~-\dual{x^{*}-x^{k},\int_{0}^{1}\int_{0}^{1}\nabla^{2}f_{\lambda^{k}}(x^{k}+st(x^{*}-x^{k}))ds(t(x^{*}-x^{k}))dt},
		\end{aligned}
	\end{equation}
	where the first inequality comes from the fact $F(x^{*})\preceq F(x^{k})$ for all $k$.
	On the other hand, $\frac{1}{2}\nm{x-x^{k}}^{2}_{\nabla^{2}f_{\lambda^{k}}(x^{k})}$ can be equivalently expressed as
	$$\dual{x-x^{k},\int_{0}^{1}\int_{0}^{1}\nabla^{2}f_{\lambda^{k}}(x^{k})ds(t(x-x^{k}))dt}.$$
	By substituting the above relation into (\ref{e4}) with $x=x^{*}$ and $x=x^{k+1}$, respectively, we have
	\begin{equation}\label{e5}
		\begin{aligned}
			&~~~~\frac{1}{2}\nm{x^{k+1}-x^{*}}^{2}_{\nabla^{2}f_{\lambda^{k}}(x^{k})}\\
			&\leq\dual{x^{k+1}-x^{k},\int_{0}^{1}\int_{0}^{1}\left(\nabla^{2}f_{\lambda^{k}}(x^{k}+st(x^{k+1}-x^{k}))-\nabla^{2}f_{\lambda^{k}}(x^{k})\right)ds(t(x^{k+1}-x^{k}))dt}\\
			&~~~~-\dual{x^{*}-x^{k},\int_{0}^{1}\int_{0}^{1}\left(\nabla^{2}f_{\lambda^{k}}(x^{k}+st(x^{*}-x^{k}))-\nabla^{2}f_{\lambda^{k}}(x^{k})\right)ds(t(x^{*}-x^{k}))dt}.
		\end{aligned}
	\end{equation}
	Since the sequence $\{x^{k}\}$ converges to a Pareto solution $x^{*}$, there exists $K_{\epsilon}\geq K^{1}_{\epsilon}$ such that, for all $k\geq K_{\epsilon}$, 
	$$\nm{\nabla^{2}f_{\lambda^{k}}(x^{k}+st(x^{k+1}-x^{k}))-\nabla^{2}f_{\lambda^{k}}(x^{k})}\leq\epsilon,~\forall s,t\in[0,1],$$
	and
	$$\nm{\nabla^{2}f_{\lambda^{k}}(x^{k}+st(x^{*}-x^{k}))-\nabla^{2}f_{\lambda^{k}}(x^{k})}\leq\epsilon,~\forall s,t\in[0,1].$$
	Substituting these two bounds into (\ref{e5}), we obtain
	$$\frac{1}{2}\nm{x^{k+1}-x^{*}}^{2}_{\nabla^{2}f_{\lambda^{k}}(x^{k})}\leq\frac{\epsilon}{2}\nm{x^{k+1}-x^{k}}^{2}+\frac{\epsilon}{2}\nm{x^{*}-x^{k}}^{2}.$$
	This, together with $\mu$-strong convexity of $f_{i}$, implies
	\begin{equation}\label{e6}
		\mu\nm{x^{k+1}-x^{*}}^{2}\leq\epsilon\nm{x^{k+1}-x^{k}}^{2}+\epsilon\nm{x^{*}-x^{k}}^{2}.
	\end{equation}
	Through direct calculation, we have
	\begin{align*}
		&~~~~\epsilon\nm{x^{k+1}-x^{k}}^{2}+\epsilon\nm{x^{*}-x^{k}}^{2}\\
		&\geq\mu\nm{x^{k+1}-x^{*}}^{2}\\
		&=\mu\nm{x^{k+1}-x^{k}+x^{k}-x^{*}}^{2}\\
		&\geq \mu\nm{x^{k+1}-x^{k}}^{2} + \mu\nm{x^{k}-x^{*}}^{2}-2\mu\nm{x^{k+1}-x^{k}}\nm{x^{k}-x^{*}}.
	\end{align*}
	Rearranging and dividing by $\nm{x^{k}-x^{*}}^{2}$, we have
	$$(\mu-\epsilon)\tau_{k}^{2}-2\mu\tau_{k}+\mu-\epsilon\leq0,$$
	where $\tau_{k}:=\frac{\nm{x^{k+1}-x^{k}}}{\nm{x^{k}-x^{*}}}$. Since $\epsilon\leq(1-\sigma)\mu$, we deduce that $\tau_{k}\in\left[\frac{\mu-\sqrt{2\mu\epsilon-\epsilon^{2}}}{\mu-\epsilon},\frac{\mu+\sqrt{2\mu\epsilon-\epsilon^{2}}}{\mu-\epsilon}\right]$.
	Substituting $\tau_{k}$ in to relation (\ref{e6}), we derive that
	$$\nm{x^{k+1}-x^{*}}\leq\sqrt{\frac{\epsilon(1+\tau_{k}^{2})}{\mu}}\nm{x^{k}-x^{*}}.$$
	Furthermore, since $\epsilon$ tends to $0$ when $k$ tends to infinity, it follows that
	$$\lim\limits_{k\rightarrow\infty}\tau_{k}\in\lim\limits_{\epsilon\rightarrow0}\left[\frac{\mu-\sqrt{2\mu\epsilon-\epsilon^{2}}}{\mu-\epsilon},\frac{\mu+\sqrt{2\mu\epsilon-\epsilon^{2}}}{\mu-\epsilon}\right]=\{1\}.$$
	We use the relation to get
	$$\lim\limits_{k\rightarrow\infty}\frac{\nm{x^{k+1}-x^{*}}}{\nm{x^{k}-x^{*}}}=0.$$
	This concludes the proof.
\end{proof}

\subsection{Quadratic convergence}
The additional assumption of Lipschitz continuity
of the Hessian $\nabla^{2} f_{i}$ for $i\in[m]$ guarantees a quadratic convergence rate of the NPGMO, as we will now demonstrate.
\begin{theorem}
	Suppose $f_{i}$ is strongly convex with module $\mu>0$, and its Hessian is Lipschitz continuous with constant $L_{2}$ for $i\in[m]$. Let $\{x^{k}\}$ be the sequence produced by Algorithm \ref{npgmo}. Then, for all $\epsilon>0$, there exists $K_{\epsilon}>0$ such that $$\nm{x^{k+1}-x^{*}}\leq\frac{(2\tau_{k}+1)L_{2}}{3\mu-\tau_{k}L\nm{x^{k}-x^{*}}}\nm{x^{k}-x^{*}}^{2}$$ holds for all $k\geq K_{\epsilon}$, where  $\tau_{k}:=\frac{\nm{x^{k+1}-x^{k}}}{\nm{x^{k}-x^{*}}}\in\left[\frac{\mu-\sqrt{2\mu\epsilon-\epsilon^{2}}}{\mu-\epsilon},\frac{\mu+\sqrt{2\mu\epsilon-\epsilon^{2}}}{\mu-\epsilon}\right]$. Furthermore, the sequence $\{x^{k}\}$ converges quadratically to $x^{*}$.
\end{theorem}
\begin{proof}
	Drawing from the arguments presented in the proof of Theorem \ref{T4}, we can establish that for any $0<\epsilon\leq(1-\sigma)\mu$, there exists a threshold $K_{\epsilon}>0$ such that, for all $k\geq K_{\epsilon}$,
	$$t_{k}=1$$ and  $$\tau_{k}\in\left[\frac{\mu-\sqrt{2\mu\epsilon-\epsilon^{2}}}{\mu-\epsilon},\frac{\mu+\sqrt{2\mu\epsilon-\epsilon^{2}}}{\mu-\epsilon}\right],$$ 
	where $\tau_{k}=\frac{\nm{x^{k+1}-x^{k}}}{\nm{x^{k}-x^{*}}}$.
	Utilizing relation (\ref{e5}), we deduce that
	\begin{small}
		\begin{align*}
			&~~~~\frac{1}{2}\nm{x^{k+1}-x^{*}}^{2}_{\nabla^{2}f_{\lambda^{k}}(x^{k})}\\
			&\leq\dual{x^{k+1}-x^{*}+x^{*}-x^{k},\int_{0}^{1}\int_{0}^{1}\left(\nabla^{2}f_{\lambda^{k}}(x^{k}+st(x^{k+1}-x^{k}))-\nabla^{2}f_{\lambda^{k}}(x^{k})\right)ds(t(x^{k+1}-x^{*}+x^{*}-x^{k}))dt}\\
			&~~~~-\dual{x^{*}-x^{k},\int_{0}^{1}\int_{0}^{1}\left(\nabla^{2}f_{\lambda^{k}}(x^{k}+st(x^{*}-x^{k}))-\nabla^{2}f_{\lambda^{k}}(x^{k})\right)ds(t(x^{*}-x^{k}))dt}\\
			&=\dual{x^{k+1}-x^{*},\int_{0}^{1}\int_{0}^{1}\left(\nabla^{2}f_{\lambda^{k}}(x^{k}+st(x^{k+1}-x^{k}))-\nabla^{2}f_{\lambda^{k}}(x^{k})\right)ds(t(x^{k+1}-x^{*}))dt}\\
			&~~~~+2\dual{x^{k+1}-x^{*},\int_{0}^{1}\int_{0}^{1}\left(\nabla^{2}f_{\lambda^{k}}(x^{k}+st(x^{k+1}-x^{k}))-\nabla^{2}f_{\lambda^{k}}(x^{k})\right)ds(t(x^{*}-x^{k}))dt}\\
			&~~~~+\dual{x^{*}-x^{k},\int_{0}^{1}\int_{0}^{1}\left(\nabla^{2}f_{\lambda^{k}}(x^{k}+st(x^{k+1}-x^{k}))-\nabla^{2}f_{\lambda^{k}}(x^{k}+st(x^{*}-x^{k}))\right)ds(t(x^{*}-x^{k}))dt}
		\end{align*}
			\begin{align*}
			&\leq\int_{0}^{1}\int_{0}^{1}L_{2}st^{2}\nm{x^{k+1}-x^{k}}dsdt\nm{x^{k+1}-x^{*}}^{2} +2\int_{0}^{1}\int_{0}^{1}L_{2}st^{2}\nm{x^{k+1}-x^{k}}dsdt\nm{x^{k+1}-x^{*}}\nm{x^{k}-x^{*}}\\
			&~~~~+\int_{0}^{1}\int_{0}^{1}L_{2}st^{2}\nm{x^{k+1}-x^{*}}dsdt\nm{x^{k}-x^{*}}^{2}\\
			&=\frac{L_{2}}{6}\nm{x^{k+1}-x^{*}}\left(\nm{x^{k+1}-x^{k}}\nm{x^{k+1}-x^{*}}+2\nm{x^{k+1}-x^{k}}\nm{x^{k}-x^{*}}+\nm{x^{k}-x^{*}}^{2}\right)\\
			&=\frac{L_{2}}{6}\nm{x^{k+1}-x^{*}}\left(\tau_{k}\nm{x^{k}-x^{*}}\nm{x^{k+1}-x^{*}}+(2\tau_{k}+1)\nm{x^{k}-x^{*}}^{2}\right),
		\end{align*}
	\end{small}
where the second inequality can be attributed to the Lipschitz continuity of $\nabla^{2}f_{i}$ for $i\in[m]$,  while the final equality originates from the definition of $\tau_{k}$. By reordering terms and leveraging the $\mu$-strong convexity of $f_{i}$, we derive
	$$\nm{x^{k+1}-x^{*}}\leq\frac{(2\tau_{k}+1)L_{2}}{3\mu-\tau_{k}L_{2}\nm{x^{k}-x^{*}}}\nm{x^{k}-x^{*}}^{2}.$$
	This, together with the convergence of ${x^{k}}$ to $x^{*}$ and the limit $\lim\limits_{k\rightarrow\infty}\tau_{k}=1$, leads to
	$$\lim\limits_{k\rightarrow\infty}\frac{\nm{x^{k+1}-x^{*}}}{\nm{x^{k}-x^{*}}^{2}}=\frac{L_{2}}{\mu}.$$
	This concludes the proof.
\end{proof}

\section{Conclusions}
In this paper, we have demonstrated the appealing convergence properties of NPGMO, including quadratic termination, locally superlinear convergence, and locally quadratic convergence. These results were established within a unified framework, which can potentially serve as a template for analyzing second-order methods for MOPs. 

\bibliographystyle{abbrv}
\bibliography{references}

\begin{acknowledgements}
This work was funded by the Major Program of the National Natural Science Foundation of China [grant numbers 11991020, 11991024]; the National Natural Science Foundation of China [grant numbers 11971084, 12171060]; NSFC-RGC (Hong Kong) Joint Research Program [grant number 12261160365]; the Team Project of Innovation Leading Talent in Chongqing [grant number CQYC20210309536]; the Natural Science Foundation of Chongqing [grant number ncamc2022-msxm01]; and Foundation of Chongqing Normal University [grant numbers 22XLB005, 22XLB006].
\end{acknowledgements}

\end{document}